\documentclass[11pt]{article}
\usepackage[a4paper,margin=1in]{geometry}
\usepackage[T1]{fontenc}
\usepackage[utf8]{inputenc}
\usepackage{lmodern}
\usepackage{amsmath,amssymb,amsthm,mathtools}
\usepackage{microtype}
\usepackage[colorlinks=true,linkcolor=blue,citecolor=blue,urlcolor=blue]{hyperref}

\newtheorem{theorem}{Theorem}
\newtheorem{lemma}{Lemma}
\newtheorem{remark}{Remark}

\newcommand{\Zi}{\mathbb Z[i]}
\newcommand{\Z}{\mathbb Z}
\newcommand{\Q}{\mathbb Q}
\newcommand{\N}{\mathbb N}
\newcommand{\Squares}{\Box}
\newcommand{\calY}{\mathcal Y}

\title{An AI Proof of 18-Variable Undecidability for Diophantine Equations over $\mathbb Z[i]$}
\author{Yuchen Ding and Junfeng Li}
\date{}

\begin{document}
\maketitle

\begin{abstract}
This paper presents an AI proof that there is no algorithm deciding whether a polynomial equation over the Gaussian integers in $18$ unknowns has a solution.  The proof improves the $20$-unknown theorem of Matiyasevich and Sun.  It follows their rationality criterion and integer test, but saves two variables: the clearing-denominator variable is avoided by imposing two integer conditions, and the remaining nonzero condition is absorbed into the relation-combining lemma by the one-variable gadget $(2R+1)(3R+1)$.
\end{abstract}

\noindent\textbf{2020 Mathematics Subject Classification.} Primary 11U05, 03D35; Secondary 03D25, 11D99, 11R11.

\medskip
\noindent\textbf{Key words and phrases.} Hilbert's Tenth Problem, Diophantine equations, Gaussian integers, undecidability, fixed number of variables, AI proof.

\section{Introduction}

Hilbert's Tenth Problem over a commutative ring asks for an algorithm deciding whether a polynomial equation over that ring has a solution.  Over $\Z$ the answer is negative, by the theorem of Matiyasevich, Davis, Putnam and Robinson; see \cite{DPR1961,Matiyasevich1970,MatiyasevichBook}.  A natural refinement asks how many unknowns are needed for undecidability.  We shall use a fixed-variable form over $\Z$ due to Sun \cite{Sun2021}.
Let $\mathbb Z[i]$ denote the ring of Gaussian integers.  Matiyasevich and Sun proved that undecidability over $\Zi$ holds for equations with $20$ unknowns \cite{MS20}.  Their proof starts from Sun's fixed-variable theorem, uses a Diophantine test for rational integers inside $\Zi$, and then applies a rationality criterion over $\mathbb Q(i)$.

The AI proof keeps the same framework, but makes two small changes.  For variables $Z_1,\ldots,Z_{10}$ put
\[
\calY=2\prod_{k=1}^{10}(3Z_k+1).
\]
Instead of adding a new variable to clear denominators, we impose
\[
\calY^{10}\in\Z,
\qquad
\calY^{11}+\sum_{k=1}^{10}Z_k\calY^{10-k}\in\Z.
\]
Since $\calY\ne0$, these conditions give
\[
\calY+\sum_{k=1}^{10}\frac{Z_k}{\calY^k}\in\Q,
\]
and the rationality criterion of \cite{MS20} then forces all $Z_k$ to be rational integers.

The second change concerns the remaining nonzero condition.  We use
\[
Q(R)=(2R+1)(3R+1).
\]
This polynomial never vanishes on $\Zi$, and for every nonzero rational integer $n$ there is an integer $R$ with $n\mid Q(R)$.  Thus the nonzero condition for $Z_{10}V_0V_1$ can be put into the relation-combining lemma without introducing the extra variable used in the earlier construction.

The variables are therefore
\[
Z_1,\ldots,Z_{10},\ M,R,\ V_0,X_0,Y_0,\ V_1,X_1,Y_1,
\]
namely $10+1+1+3+3=18$ variables.  The proof gives the following theorem.

Throughout the paper
\[
\N=\{0,1,2,\ldots\},
\]
and $\Z$ and $\Q$ are regarded as subrings of $\Zi$ and $\Q(i)$ in the usual way.

\begin{theorem}\label{thm:main}
There is no algorithm to decide, for an arbitrary polynomial
\[
F(X_1,\ldots,X_{18})\in\Z[X_1,\ldots,X_{18}],
\]
whether the equation
\[
F(X_1,\ldots,X_{18})=0
\]
has a solution in $\Zi^{18}$.
\end{theorem}

{\bf Acknowledgement.} We thank Prof. Z.-W. Sun for his interest in this article.

\section{Auxiliary facts}

Let
\[
\Squares=\{\xi^2:\xi\in \Zi\}.
\]
We shall use the following facts from \cite{MS20}.

\begin{lemma}[conjunction over $\Zi$]\label{lem:conjunction}
This is \cite[Lemma~3.2]{MS20}. For any $X,Y\in \Zi$,
\[
X=0\ \wedge\ Y=0
\quad\Longleftrightarrow\quad
X^2+2Y^2=0.
\]
Consequently, any finite system of polynomial equations over $\Zi$ can be combined into a single polynomial equation, without introducing new variables.  In particular, if the original equations have coefficients in $\Z$, then the combined equation also has coefficients in $\Z$.
\end{lemma}

\begin{lemma}[combining two square conditions and one divisibility condition]\label{lem:combining}
This is \cite[Lemma~3.1]{MS20}. Let $A_1,A_2,S,T\in \Zi$ with $A_1\ne A_2$ and $S\ne0$.  Define
\[
f(A_1,A_2,S,T,M)
=(T-MS)^4-2(A_1+A_2)S^2(T-MS)^2+(A_1-A_2)^2S^4.
\]
Then
\[
A_1\in\Squares\ \wedge\ A_2\in\Squares\ \wedge\ S\mid T
\]
if and only if there exists $M\in\Zi$ such that
\[
f(A_1,A_2,S,T,M)=0.
\]
\end{lemma}

\begin{lemma}[integer test in $\Zi$]\label{lem:integer-test}
This is \cite[Lemma~3.3]{MS20}. For $U\in\Zi$, we have $U\in\Z$ if and only if there exist $V,X,Y\in\Zi$ with $V\ne0$ such that
\begin{equation}\label{eq:int-test-one}
4\bigl(2V(2(2U+1)^2+1)-Y\bigr)^2-3Y^2-1=0
\end{equation}
and
\begin{equation}\label{eq:int-test-two}
3Y^2(2U+1-XY)^2+1\in\Squares.
\end{equation}
Moreover, if $U\in\Z$, then $V,X,Y$ may be chosen in $\Z$.
\end{lemma}

\begin{lemma}[rationality criterion]\label{lem:rationality}
This is \cite[Theorem~1.2]{MS20} in the case $K=\mathbb Q(i)$. Let $x_1,\ldots,x_n\in\Zi$ and put
\[
y=2\prod_{k=1}^n(3x_k+1).
\]
Then
\[
y+\sum_{k=1}^n\frac{x_k}{y^k}\in\Q
\quad\Longleftrightarrow\quad
x_1,\ldots,x_n\in\Z.
\]
\end{lemma}

We also use Sun's fixed-variable form of Matiyasevich's theorem over $\Z$: there exists a recursively enumerable but nonrecursive set $\mathcal A\subseteq\N$ and a polynomial
\[
P(Z_0,Z_1,\ldots,Z_{10})\in\Z[Z_0,Z_1,\ldots,Z_{10}]
\]
such that, for every $a\in\N$,
\begin{equation}\label{eq:sun}
a\in\mathcal A
\quad\Longleftrightarrow\quad
\exists z_1,\ldots,z_{10}\in\Z\quad
\bigl(P(a,z_1,\ldots,z_{10})=0\ \wedge\ z_{10}\ne0\bigr).
\end{equation}
This is the precise fixed-variable input used in \cite[proof of Theorem~1.1]{MS20}, based on \cite[Theorem~1.1(ii)]{Sun2021}.

We need one elementary observation which replaces the two-variable nonzero encoding used in the 19-variable version.

\begin{lemma}[one-variable nonzero-divisibility gadget]\label{lem:gadget}
For $R\in\Zi$, put
\[
Q(R)=(2R+1)(3R+1).
\]
Then $Q(R)\ne0$ for every $R\in\Zi$.  Moreover, for every nonzero integer $n\in\Z$, there exists $R\in\Z$ such that
\[
n\mid Q(R)
\]
in $\Z$, and hence also in $\Zi$.
\end{lemma}

\begin{proof}
If $2R+1=0$ for some $R\in\Zi$, then comparing real parts gives an integer solution of $2a+1=0$, which is impossible.  Similarly, $3R+1=0$ has no solution in $\Zi$.  Hence $Q(R)\ne0$ for all $R\in\Zi$.

Let $n\in\Z\setminus\{0\}$ and put $m=|n|$.  If $m=1$, any $R\in\Z$ works.  Suppose $m>1$, and write
\[
m=2^a3^b\prod_{\ell\in\mathcal P}\ell^{e_\ell},
\]
where every $\ell\in\mathcal P$ is a prime distinct from $2$ and $3$.  Choose congruences as follows:
\[
3R+1\equiv0\pmod {2^a}\quad (a>0),
\]
\[
2R+1\equiv0\pmod {3^b}\quad (b>0),
\]
and, for each $\ell\in\mathcal P$,
\[
2R+1\equiv0\pmod {\ell^{e_\ell}}.
\]
Each displayed congruence is solvable because the relevant coefficient is invertible modulo the indicated prime power.  The moduli with exponent $0$ are omitted.  The remaining moduli are pairwise coprime, so the Chinese remainder theorem gives an integer $R$ satisfying all of them simultaneously.  Then every prime-power divisor of $m$ divides one of the two factors $2R+1$ or $3R+1$, and hence $m\mid Q(R)$.  Thus $n\mid Q(R)$.
\end{proof}

\section{Proof of the main theorem}

\begin{proof}[Proof of Theorem \ref{thm:main}]
Fix the recursively enumerable but nonrecursive set $\mathcal A\subseteq\N$ and the polynomial $P$ in \eqref{eq:sun}.  We shall show that membership in $\mathcal A$ can be represented by one polynomial equation over $\Zi$ in $18$ unknowns.  This will imply the theorem, since a decision algorithm for solvability of arbitrary $18$-variable equations over $\Zi$ would decide membership in $\mathcal A$.

Fix $a\in\N$.  For variables $Z_1,\ldots,Z_{10}\in\Zi$, put
\begin{equation}\label{eq:Y-def}
\calY=2\prod_{k=1}^{10}(3Z_k+1).
\end{equation}
Since $3Z_k+1\ne0$ for every $Z_k\in\Zi$ (writing $Z_k=a+bi$ would otherwise give $3b=0$ and $3a+1=0$), we have $\calY\ne0$.  Define
\begin{equation}\label{eq:T-W-def}
T=\sum_{k=1}^{10} Z_k\calY^{10-k},
\qquad
W=\calY^{11}+T.
\end{equation}
Then
\begin{equation}\label{eq:ratio}
\frac{W}{\calY^{10}}
=
\calY+\sum_{k=1}^{10}\frac{Z_k}{\calY^k}.
\end{equation}

The key point is that we do not introduce an auxiliary integer multiplier.  Instead, we force the two quantities
\[
\calY^{10}\quad\text{and}\quad W
\]
to be rational integers.  If this happens, then \eqref{eq:ratio} belongs to $\Q$, because the denominator $\calY^{10}$ is a nonzero rational integer; Lemma \ref{lem:rationality} then gives
\[
Z_1,\ldots,Z_{10}\in\Z.
\]

We now express the two integer conditions by Lemma \ref{lem:integer-test}.  Let
\[
U_0=\calY^{10},\qquad U_1=W.
\]
For $j=0,1$, introduce variables $V_j,X_j,Y_j\in\Zi$ and set
\begin{equation}\label{eq:E-j}
E_j=4\bigl(2V_j(2(2U_j+1)^2+1)-Y_j\bigr)^2-3Y_j^2-1,
\end{equation}
and
\begin{equation}\label{eq:B-j}
B_j=3Y_j^2(2U_j+1-X_jY_j)^2+1.
\end{equation}
Thus, by Lemma \ref{lem:integer-test}, the condition $U_j\in\Z$ is equivalent to
\[
\exists V_j,X_j,Y_j\in\Zi\quad
\bigl(V_j\ne0\ \wedge\ E_j=0\ \wedge\ B_j\in\Squares\bigr).
\]
The nonzero conditions on $V_0$ and $V_1$ will be enforced below through the single product $Z_{10}V_0V_1$, rather than by separate equations.

We shall combine the two square conditions and the nonzero condition using Lemma \ref{lem:combining}.  Define
\begin{equation}\label{eq:C-def}
C_0=9B_0,
\qquad
C_1=B_1.
\end{equation}
Then, in the quotient $\Zi/(3)$,
\[
C_0\equiv0\pmod 3,
\qquad
C_1\equiv1\pmod 3,
\]
because $B_1=3Y_1^2(2U_1+1-X_1Y_1)^2+1$.  Since the rational prime $3$ is a Gaussian prime, $\Zi/(3)$ is an integral domain and in particular $1\ne0$ there.  Hence $C_0\ne C_1$ in $\Zi$.
Moreover,
\begin{equation}\label{eq:C0-B0-equivalence}
C_0\in\Squares
\quad\Longleftrightarrow\quad
B_0\in\Squares.
\end{equation}
Indeed, if $B_0=\beta^2$, then $C_0=(3\beta)^2$.  Conversely, if $C_0=\eta^2$ for some $\eta\in\Zi$, then
\[
\eta^2=9B_0.
\]
The rational prime $3$ is a Gaussian prime, since $3\equiv3\pmod4$.  Hence $3\mid\eta$, say $\eta=3\eta'$.  Substituting gives
\[
9(\eta')^2=9B_0.
\]
Since $9\ne0$ and $\Zi$ is an integral domain, cancellation gives $B_0=(\eta')^2\in\Squares$.

Put
\[
N=Z_{10}V_0V_1,
\qquad
Q(R)=(2R+1)(3R+1).
\]
The already proved inequality $C_0\ne C_1$ supplies the hypothesis $A_1\ne A_2$ in Lemma \ref{lem:combining}.  Instead of imposing a separate nonzero equation, we impose the single relation-combining equation
\begin{equation}\label{eq:f-combine-18}
f(C_0,C_1,N,Q(R),M)=0.
\end{equation}
This equation will force $N\ne0$ automatically.  Indeed, if $N=0$, then by the definition of $f$ we get
\[
f(C_0,C_1,0,Q(R),M)=Q(R)^4.
\]
By Lemma \ref{lem:gadget}, $Q(R)\ne0$, and since $\Zi$ is an integral domain, $Q(R)^4\ne0$.  Hence \eqref{eq:f-combine-18} cannot hold with $N=0$.

We claim that, for the fixed parameter $a\in\N$, the following finite system of equations in the $18$ unknowns
\[
Z_1,\ldots,Z_{10},\ M,R,\ V_0,X_0,Y_0,\ V_1,X_1,Y_1
\]
has a solution over $\Zi$ if and only if $a\in\mathcal A$:
\begin{align}
P(a,Z_1,\ldots,Z_{10})&=0, \label{eq:system-P}\\
E_0&=0, \label{eq:system-E0}\\
E_1&=0, \label{eq:system-E1}\\
f(C_0,C_1,Z_{10}V_0V_1,Q(R),M)&=0. \label{eq:system-f}
\end{align}
Here $\calY,T,W,U_j,E_j,B_j,C_j$ are the polynomials defined in \eqref{eq:Y-def}--\eqref{eq:C-def}.

First suppose that $a\in\mathcal A$.  By \eqref{eq:sun}, choose $z_1,\ldots,z_{10}\in\Z$ such that
\[
P(a,z_1,\ldots,z_{10})=0,
\qquad
z_{10}\ne0.
\]
Then $\calY\in\Z$ and $W\in\Z$, hence $U_0=\calY^{10}\in\Z$ and $U_1=W\in\Z$.  By the last clause of Lemma \ref{lem:integer-test}, for $j=0,1$ we may choose $V_j,X_j,Y_j\in\Z$ with $V_j\ne0$ such that $E_j=0$ and $B_j\in\Squares$.  Then $C_0,C_1\in\Squares$.  Moreover, since $\Zi$ is an integral domain and $z_{10},V_0,V_1$ are all nonzero,
\[
N=z_{10}V_0V_1\in\Z\setminus\{0\}.
\]
By Lemma \ref{lem:gadget}, choose $R\in\Z$ such that $N\mid Q(R)$.  The hypotheses of Lemma \ref{lem:combining} now hold: $C_0\ne C_1$ and $N\ne0$.  Applying the lemma with
\[
A_1=C_0,
\quad
A_2=C_1,
\quad
S=N,
\quad
T=Q(R),
\]
gives some $M\in\Zi$ satisfying \eqref{eq:system-f}.  Thus the system \eqref{eq:system-P}--\eqref{eq:system-f} has a solution over $\Zi$.

Conversely, suppose that the system \eqref{eq:system-P}--\eqref{eq:system-f} has a solution over $\Zi$.  Put again
\[
N=Z_{10}V_0V_1.
\]
By the observation following \eqref{eq:f-combine-18}, equation \eqref{eq:system-f} implies $N\ne0$.  Since $\Zi$ is an integral domain and $N=Z_{10}V_0V_1$, each factor is nonzero; hence
\[
Z_{10}\ne0,
\qquad
V_0\ne0,
\qquad
V_1\ne0.
\]
Since $C_0\ne C_1$ for every assignment and $N\ne0$, Lemma \ref{lem:combining}, applied to \eqref{eq:system-f} with $S=N$ and $T=Q(R)$, gives
\[
C_0\in\Squares,
\qquad
C_1\in\Squares.
\]
By \eqref{eq:C0-B0-equivalence}, $B_0\in\Squares$, and since $C_1=B_1$, also $B_1\in\Squares$.
Together with $E_0=E_1=0$ and $V_0,V_1\ne0$, Lemma \ref{lem:integer-test} yields
\[
U_0=\calY^{10}\in\Z,
\qquad
U_1=W\in\Z.
\]
As observed above, $\calY\ne0$.  Hence $\calY^{10}\ne0$; since $\calY^{10}\in\Z$ and $W\in\Z$, the quotient has nonzero rational-integer denominator.  Therefore
\[
\calY+\sum_{k=1}^{10}\frac{Z_k}{\calY^k}
=\frac{W}{\calY^{10}}
\in\Q.
\]
By Lemma \ref{lem:rationality}, applied with $x_k=Z_k$, we conclude that
\[
Z_1,\ldots,Z_{10}\in\Z.
\]
Now \eqref{eq:system-P} and $Z_{10}\ne0$ show, via \eqref{eq:sun}, that $a\in\mathcal A$.

Thus $a\in\mathcal A$ if and only if the system \eqref{eq:system-P}--\eqref{eq:system-f} has a solution over $\Zi$ in the displayed $18$ unknowns.  By Lemma \ref{lem:conjunction}, this finite system can be replaced by a single polynomial equation
\[
F_a(X_1,\ldots,X_{18})=0
\]
with coefficients in $\Z$, without adding variables.  The coefficients remain in $\Z$ because all displayed equations have coefficients in $\Z$ and the operation $X^2+2Y^2$ preserves this coefficient ring.  The construction of $F_a$ from $a$ is effective: substitute the fixed parameter $a$ into the fixed polynomial $P$, write the four displayed equations as polynomials in the $18$ variables, and apply Lemma~\ref{lem:conjunction} finitely many times.

If there were an algorithm deciding solvability over $\Zi$ for arbitrary polynomial equations in $18$ unknowns with integer coefficients, then applying it to $F_a$ would decide whether $a\in\mathcal A$.  This contradicts the nonrecursiveness of $\mathcal A$.  Hence no such algorithm exists.
\end{proof}

\begin{remark}
The first saving comes from replacing the original auxiliary integer multiplier $t$ by the two integer conditions
\[
\calY^{10}\in\Z,
\qquad
\calY^{11}+\sum_{k=1}^{10}Z_k\calY^{10-k}\in\Z.
\]
The second saving comes from replacing the two-variable nonzero equation by the one-variable polynomial $Q(R)=(2R+1)(3R+1)$ inside the relation-combining lemma.  Thus the argument avoids both the old variable $t$ and one of the two variables previously used to force nonvanishing.  Within this particular framework the variable count is
\[
10+3+3+1+1=18,
\]
coming respectively from Sun's ten variables, the two three-variable integer tests, the combining variable $M$, and the one-variable gadget $R$.  No absolute optimality statement is claimed here; lowering the count to $17$ would require an additional improvement, such as a smaller integer test, a shared witness for the two integer conditions, a stronger fixed-variable input, or a different combining mechanism.
\end{remark}

\end{document}